\documentclass[12pt, reqno,fleqn]{amsart}
\allowdisplaybreaks[1]
\usepackage{amsmath}
\usepackage{amssymb}
\usepackage{amsfonts}
\usepackage{verbatim}
\usepackage[usenames]{color}
\usepackage{hyperref}
\makeindex

 \newtheorem{theorem}{Theorem}[section]
 \newtheorem{corollary}[theorem]{Corollary}

\newtheorem{claim}[theorem]{Claim}
\theoremstyle{definition}

\theoremstyle{remark}

\newtheorem*{claim*}{Claim}
\newtheorem{fact*}{Fact}

\newcommand{\til}{\raise.17ex\hbox{$\scriptstyle\mathtt{\sim}$}}

\newcommand\beq{\begin{equation}}

\newcommand\eeq{\end{equation}}

\newcommand{\bbm}{\left[ \begin{smallmatrix}}
\newcommand{\ebm}{\end{smallmatrix} \right]}
\newcommand{\bpm}{\left( \begin{smallmatrix}}
\newcommand{\epm}{\end{smallmatrix} \right)}
\numberwithin{equation}{section}

\newlength{\Mheight}
\newlength{\cwidth}

\def\McCarthy{M\raise.45ex\hbox{c}Carthy }
\def\McCarthyc{M\raise.45ex\hbox{c}Carthy, }

\title[Note on L\"owner's theorem]{Note on L\"owner's theorem on matrix monotone functions in several commuting variables of Agler, McCarthy and Young}
\author{
J. E. Pascoe
}
\thanks{
Partially supported by National Science Foundation Grant DMS 1361720}
\date{\today}

\subjclass[2010]{Primary	47A63  Secondary   	32A40}

\begin{document}

\begin{abstract}
In this brief note,
we show that the hypotheses of L\"owner's theorem on matrix monotonicity in several commuting variables as proved by
Agler, \McCarthy and Young
can be significantly relaxed. Specifically, we extend their theorem from continuously differentiable locally
matrix monotone functions to arbitrary locally matrix monotone functions using mollification techniques.
\end{abstract}
\maketitle


\section{Introduction}

A function $f: (a,b) \rightarrow \mathbb{R}$ is \emph{matrix monotone} if
	$$A \leq B \Rightarrow f(A) \leq f(B)$$
for all $A, B$ self adjoint matrices with spectrum in $(a,b),$ where $A\leq B$ means $B-A$ is positive semidefinite.
In 1934\cite{lo34}, Charles L\"owner showed that if $f: (a,b) \rightarrow \mathbb{R}$ is matrix monotone, then $f$ analytically continues
to the upper half plane $\Pi \subset \mathbb{C}$ as a map $f: \Pi \cup (a,b) \rightarrow \overline{\Pi}.$

Agler, \McCarthy and Young extended L\"owner's theorem to several commuting
variables for the class of locally matrix monotone functions.
Let $E$ be an open subset of $\mathbb{R}^d.$
Let $CSAM^d_n(E)$ denote the $d$-tuples of commuting self-adjoint matrices of size $n$ with joint spectrum contained in $E$. (That is,
if you jointly diagonalize an element of $CSAM^d_n(E)$ it should look like a direct sum of elements of $E$.)
A \emph{locally matrix monotone function} is a function
$f: E \rightarrow \mathbb{R}$ so that for every $n,$ on every $C^1$ curve $\gamma: [0,1] \rightarrow CSAM^d_n(E)$ such that 
$\gamma'(t)_i \geq 0$ for all $i$ and all $t \in [0,1],$
	\beq\label{monotoneCondition} t_1 \leq t_2 \Rightarrow f(\gamma(t_1)) \leq f(\gamma(t_2)).\eeq

We generalize the following result.
\begin{theorem}[Agler, \McCarthyc Young \cite{amyloew}]\label{amyloewthm}
	Let $E$ be an open subset of $\mathbb{R}^d.$ A $C^1$ function $f: E \rightarrow \mathbb{R}$ is locally matrix monotone
	if and only if $f$ is analytic and $f$ analytically continues to $\Pi^d$
	as a map $f: E \cup \Pi^d \rightarrow \overline{\Pi}.$
\end{theorem}

We show that the assumption that $f$ is $C^1$ can be dropped 
as in L\"owner's theorem for one variable. That is, we prove the following theorem.
\begin{theorem}\label{mainresult}
	Let $E$ be an open subset of $\mathbb{R}^d.$
	Let $f: E \rightarrow \mathbb{R}$ be a locally matrix monotone function, then $f$ is analytic.	
\end{theorem}

So as an immediate corollary, we obtain the following.
\begin{corollary}
	Let $E$ be an open subset of $\mathbb{R}^d.$ A function $f: E \rightarrow \mathbb{R}$ is locally matrix monotone
	if and only if $f$ is analytic and $f$ analytically continues to $\Pi^d$
	as a map $f: E \cup \Pi^d \rightarrow \overline{\Pi}.$
\end{corollary}

\section{Proof of the result}
	We fix the convention that for $a, b \in \mathbb{R}^d,$
$a< b$ means that $b_i-a_i > 0$ for all $i$  and $a\leq  b$ means that $b_i-a_i \geq 0.$
Furthermore for $a\in \mathbb{R}^d$ we define $\|a\| = \|a\|_{\infty} = \sup_i |a_i|.$
With the ordering induced by $\leq,$ locally matrix monotone
functions define on a ball with respect to the above norm are monotone in the more
conventional sense that
	\beq \label{monoone}
	a \leq b \Rightarrow f(a) \leq f(b).\eeq

To prove Theorem \ref{mainresult} we mollify a locally matrix monotone function $f$ to get a smooth matrix monotone function
and use the analytic continuations of those from Theorem \ref{amyloewthm} to derive an analytic continuation for $f$ itself. 

	Let $E$ be an open subset of $\mathbb{R}^d.$
	Fix $f: E \rightarrow \mathbb{R}$ be a locally matrix monotone function.
	Fix $K$ an open set such that $\overline{K}$ is compact and $\overline{K} \subset E$. Note that it is sufficient
	to prove that $f$ is analytic on each such $K$ since analyticity is a local property. Let $\mu$ be Lebesgue measure.
	Let $\phi: \mathbb{R}^d\rightarrow \mathbb{R}$ be a nonnegative smooth function with compact support so that
	$\int_{\mathbb{R}^d} \phi(x) d\mu(x) = 1.$
	Let $\phi_t(x) = t^{-d}\phi(t^{-1}x).$
	Let $f_t = \phi_t * f$ where $f$ is formally extended to be $0$ off $E.$
	Let $\vec{1} = (1,1,\ldots, 1) \in \mathbb{R}^d.$

	We note $f_t$ is well-defined since
	$f$ is continuous, because, for any $x\in E$, the function $f(x + \epsilon\vec{1})$
	is continuous as a function of $\epsilon$
	by the classical L\"owner theorem and, for all $y\in E$ such that
	$\|x-y\|<\epsilon,$
		$$f(x - \epsilon\vec{1})\leq f(y) \leq f(x+\epsilon \vec{1}).$$
	Furthermore, we note $f_t \rightarrow f$ pointwise since $f$ is continuous.
	
Now, we give criteria for the mollification $f_t$ of a locally matrix monotone function to be locally matrix monotone.
\begin{claim}\label{mollifyMonotone}
	The function $f_t$ is locally matrix monotone on $K$ for sufficiently small $t.$
\end{claim}
\begin{proof}
	Let $t$ be small enough so that  
	$K + \text{Supp }\phi_t \subset E.$
	If $a \leq b,$
	$$f_t(\gamma(b)) - f_t(\gamma(a)) = \int_{\text{Supp }\phi_t} \phi_t(\xi)
	(f(\gamma(b)-\xi) - f(\gamma(a)-\xi))d\mu(\xi)\geq 0$$ since
	$$f(\gamma(b)-\xi) - f(\gamma(a)-\xi) \geq 0$$
	 because
	$\gamma_\xi =  \gamma - \xi$ is itself a path in $CSAM^d_n(E)$ such that $\gamma_\xi'(x)\geq 0$ and so by the definition of
	local matrix monotonicity, we are done.
\end{proof}

%


To show that the $f_t$ have an analytic limit as $t\rightarrow 0$ we will show that they form a normal family.
\begin{theorem}[P. {\cite[Theorem 4.4]{pascoeLOMC}} ]\label{oldresult1}
Let $F \subset \mathbb{R}^d$ be open. Let $p\in F.$
There are absolute constants $q^p_k$ so that $q^p_k$ is $O(c^k)$ for some $c>0$ so that for every
differentiable locally matrix monotone function
$f: F \rightarrow \mathbb{R},$
$$\frac{|f^{(k)}(p)[\vec{z}]|}{k!} \leq  q^p_k\|\vec{z}\|^k f'(p)[\vec{1}].$$
(Here $f^{(k)}(p)[\vec{z}] = \frac{d^k}{dt^k} f(p + t\vec{z}).$)

Namely, if $F=\{a|\|a-x\|\leq \delta\},$ there is are constants $c$ and $D$ such that
for every 
differentiable locally matrix monotone function $f:F \rightarrow \mathbb{R},$
$f$ analytically continues to all $z \in \mathbb{C}^d$ such $\|z-x\|\leq c /\delta$ and 
$$|f(z)-f(x)|\leq D\frac{\|z\|}{1-\frac{c}{\delta}\|z\|}f'(x)[\vec{1}].$$
\end{theorem}

Theorem \ref{oldresult1} implies $(f_t)_{0<t<\epsilon}$ form a normal family if
for any basic open set in the domain of $f$ we can find an $a_0$
	such that $\limsup_{t\rightarrow 0} f_t'(a_0)[\vec{1}]<\infty.$
\begin{claim}
	Let $x\in K.$
	For any $\delta >0,$
	there is an $a_0 \in K$ such that $\|x - a_0\|< \delta$ and
	$\limsup_{t\rightarrow 0} f_t'(a_0)[\vec{1}]<\infty.$
\end{claim}
\begin{proof}
Let $\hat{K}= \{a|\|a-x\|\leq \delta\}.$
Note it is sufficient to prove the claim for all sufficiently small values of
$\delta$ so that $\hat{K} \subset K.$
Let $\rho$ be the weak limit of the measures
$$\rho_\epsilon(x) = \frac{f(x + \epsilon \vec{1}) - f(x)}{\epsilon} d\mu$$ as $\epsilon \rightarrow 0$
	taken in the dual of $C_0(\hat{K})$ which exists because $f$ is locally bounded since it is monotone.
	Thus, the total variation is bounded as follows:
		\begin{align*}
		|\rho_\epsilon| & =  \frac{1}{\epsilon}\left[
		\int_{
		\hat{K}+
		\epsilon
		\vec{1}
		}
		f d\mu - \int_{\hat{K}} f d\mu \right]\\
		&\leq  \frac{1}{\epsilon}\int_{(\hat{K}+\epsilon\vec{1}) \setminus K
		\cup \hat{K} \setminus (K+\epsilon\vec{1})} |f| d\mu \\
		 &\leq \frac{
		 \sup_{(\hat{K}+\epsilon\vec{1}) \setminus K
		\cup \hat{K} \setminus (K+\epsilon\vec{1})} |f|
		 \mu((\hat{K}+\epsilon\vec{1}) \setminus \hat{K} \cup \hat{K} \setminus ( \hat{K}+\epsilon\vec{1}))}{\epsilon}\\
		 &\leq  \frac{
		 \max(|f(x+(\epsilon+\delta)\vec{1})|,|f(x-\delta\vec{1})|)
		 }{\epsilon}2d\delta^{d-1}\epsilon \\
		 & =   \max(|f(x+(\epsilon+\delta)\vec{1})|,|f(x-\delta\vec{1})|)2d\delta^{d-1}.
		\end{align*}	
	
	Note $f_t'(x)[\vec{1}] = \phi_t * \rho.$
	Now,  pick a point $a_0$ in the interior of $\hat{K}$
	so that the density of $\rho$ with
	respect to Lebesgue measure is finite. (Such a point exists by \cite[pg 99, Theorem 3.22]{fol})
	Note that $\limsup_{t\rightarrow 0} f_t'(a_0)[\vec{1}]$ is equal to the density of
	$\rho$ with respect to Lebesgue measure, so we are done.
\end{proof}

Thus, $f$ is analytic.


\printindex
\bibliography{references}
\bibliographystyle{plain}

\end{document}